\documentclass[a4paper, 10pt, DIV10, headinclude=false, footinclude=false, leqno]{scrartcl}

\usepackage{amsmath, amssymb, amsfonts}
\usepackage[amsmath, amsthm, thmmarks]{ntheorem}
\usepackage[utf8]{inputenc}
\usepackage[english]{babel}
\usepackage{url}
\usepackage{mathtools}
\usepackage{esint}

\usepackage{tikz}

\numberwithin{figure}{section}
\numberwithin{table}{section}
\numberwithin{equation}{section}

\newenvironment{abstr}[1]{ \vspace{.05in}\footnotesize
       \parindent .2in
         {\upshape\bfseries #1. }\ignorespaces}{\par\vspace{.1in}}
\newenvironment{Abstract}{\begin{abstr}{Abstract}}{\end{abstr}}
\newenvironment{keywords}{\begin{abstr}{Key words}}{\end{abstr}}
\newenvironment{AMS}{\begin{abstr}{AMS subject classifications}}{\end{abstr}}

\newtheorem{theorem}{Theorem}[section]
\newtheorem{lemma}[theorem]{Lemma}

\newtheorem{proposition}[theorem]{Proposition}

\theoremstyle{definition}
\newtheorem{definition}[theorem]{Definition}

\DeclareMathOperator{\supp}{supp}
\DeclareMathOperator{\Int}{int}
\DeclareMathOperator{\diam}{diam}
\DeclareMathOperator{\Div}{div}
\DeclareMathOperator{\curl}{curl}

\renewcommand{\Re}{\operatorname{Re}}

\newcommand{\rz}{\mathbb{R}}       
\newcommand{\cz}{\mathbb{C}}       
\newcommand\Va{\mathbf{a}}
\newcommand\Vb{\mathbf{b}}

\newcommand\Ve{\mathbf{e}}
\newcommand\Vf{\mathbf{f}}

\newcommand\Vn{\mathbf{n}}

\newcommand\Vv{\mathbf{v}}
\newcommand\Vu{\mathbf{u}}
\newcommand\Vw{\mathbf{w}}
\newcommand\Vx{\mathbf{x}}

\newcommand\Vz{\mathbf{z}}

\newcommand\VF{\mathbf{F}}
\newcommand\VH{\mathbf{H}}

\newcommand\VW{\mathbf{W}}
\newcommand\Vpsi{\boldsymbol{\psi}}
\newcommand\Vphi{\boldsymbol{\phi}}

\newcommand\CB{\mathcal{B}}

\newcommand\CG{\mathcal{G}}

\newcommand\CK{\mathcal{K}}
\newcommand\CL{\mathcal{L}}

\newcommand\CN{\mathcal{N}}

\newcommand\CT{\mathcal{T}}

\newcommand\UN{\textup{N}}

\newcommand\id{\operatorname{id}}

\allowdisplaybreaks[4]

\begin{document}

\title{Numerical homogenization for indefinite H(curl)-problems}

\author{Barbara Verf\"urth\thanks{Institut f\"ur Analysis und Numerik, Westf\"alische Wilhelms-Universit\"at M\"unster, Einsteinstr.\ 62, D-48149 M\"unster.}}

\date{}
\maketitle

\begin{Abstract}
In this paper, we present a numerical homogenization scheme for indefinite, time-harmonic Maxwell's equations involving potentially rough (rapidly oscillating) coefficients.
The method involves an $\mathbf{H}(\mathrm{curl})$-stable, quasi-local operator, which allows for a correction of coarse finite element functions such that order optimal (w.r.t.\ the mesh size) error estimates are obtained.
To that end, we extend the procedure of [D.\ Gallistl, P.\ Henning, B.\ Verf\"urth, {\em Numerical homogenization of H(curl)-problems}, arXiv:1706.02966, 2017] to the case of indefinite problems.
In particular, this requires a careful analysis of the well-posedness of the corrector problems as well as the numerical homogenization scheme.
\end{Abstract}



\section{Introduction}
Time-harmonic Maxwell's equations, which model electromagnetic wave propagation, play an essential role in many physical applications.
If the coefficients are rapidly oscillating on a fine scale as in the context of photonic crystals or metamaterials, standard discretizations suffer from bad convergence rates and a large pre-asymptotic range due to the multiscale nature, the low regularity, and the indefiniteness of the problem.

In this paper, we consider a numerical homogenization scheme to cope with the multiscale nature and the resolution condition, which couples the maximal mesh size to the frequency and is typical for indefinite wave propagation problems, see \cite{BS00}.
Analytical homogenization for locally periodic $\VH(\curl)$-problems shows that the solution can be decomposed into a macroscopic contribution (without rapid oscillations) and a fine-scale corrector, see \cite{CFS17, HOV16a, HOV16b}.
In \cite{GHV17}, this was extended beyond the periodic case and without assuming scale separation.
Using a suitable interpolation operator, the exact solution is decomposed into a coarse part, which is a good approximation in $\VH(\Div)^\prime$, and a corrector contribution, which then gives a good approximation in $\VH(\curl)$. 
Furthermore, the corrector can be quasi-localized, allowing for an efficient computation.
Analytical homogenization and other multiscale methods can also be applied to indefinite problems \cite{CFS17, HOV16a} so that it is natural to examine this extension also for the numerical homogenization of \cite{GHV17}.

The technique of numerical homogenization presented there is known as Localized Orthogonal Decomposition (LOD) and was originally proposed in \cite{MP14}.
Among many applications, we point to elliptic boundary value problems \cite{HM14}, the wave equation \cite{AH17}, mixed elements \cite{HHM16}, and in particular Helmholtz problems \cite{GP15, Pet17, OV17}.
The works on the Helmholtz equation reveal that the LOD can also reduce the so-called pollution effect.
Only a natural and reasonable resolution condition of a few degrees of freedom per wavelength is needed in the LOD and the local corrector problems have to be solved on patches which grow logarithmically with the wave number.
The crucial observation is that the bilinear form is coercive on the kernel space of a suitable interpolation operator. 
For Maxwell's equations this is not possible due to the large kernel of the curl-operator. 
However, a wave number independent inf-sup-stability of the bilinear form over the kernel of the interpolation operator is proved using a regular decomposition.
This inf-sup-stability forces us to localize the corrector problem in a non-conforming manner, which leads to additional terms in the analysis and may be of independent interest.
Still, we are able to define a well-posed localized numerical homogenization scheme which allow for order optimal (w.r.t.\ the mesh size) a priori estimates. 

The paper is organized as follows. 
Section \ref{sec:problem} introduces the model problem and the necessary notation for meshes and the interpolation operator.
We introduce an ideal numerical homogenization scheme in Section \ref{sec:LODideal}. 
We localize the corrector operator, present the resulting main scheme and its a priori analysis in Section \ref{sec:LOD}.

The notation $a\lesssim b$ denotes $a\leq Cb $ with a constant $C$ independent of the mesh size $H$, the oversampling parameter $m$ and the frequency $\omega$.
Bold face letters will indicate vector-valued quantities and all functions are complex-valued, unless explicitly mentioned.
We study the high-frequency case, i.e.\ $\omega\gtrsim1$ is assumed.

\section{Problem setting}
\label{sec:problem}
\subsection{Model problem}
Let $\Omega\subset\mathbb{R}^3$ be an open, bounded, contractible domain with polyhedral Lipschitz boundary with outer unit normal $\Vn$.
For any bounded subdomain $G\subset \Omega$, the spaces $\VH(\curl, G)$, $\VH_0(\curl, G)$ and $\VH(\Div, G)$ denote the usual curl- and div-conforming spaces; see \cite{Monk} for details.
We will omit the domain $G$ if it is equal to the full domain $\Omega$.
In addition to the standard inner product, we equip $\VH(\curl, G)$ with the following $\omega$-dependent inner product
\[(\Vv, \Vw)_{\curl, \omega, G}:=(\curl\Vv, \curl\Vw)_{L^2(G)}+\omega^2(\Vv, \Vw)_{L^2(G)}.\]

Let $\Vf\in \VH(\Div, \Omega)$ and let $\mu\in L^\infty(\Omega, \rz^{3 \times 3})$ and $\varepsilon\in L^\infty(\Omega, \rz^{3 \times 3})$ be uniformly elliptic.
For any open subset $G\subset\Omega$, we define the bilinear form $\CB_{G}: \VH(\curl,G)\times \VH(\curl,G)\to \cz$ as
\begin{equation}
\label{eq:sesquiform}
\CB_{G}(\Vv, \Vpsi):=(\mu\curl \Vv, \curl\Vpsi)_{L^2(G)}
                      -\omega^2(\varepsilon\Vv, \Vpsi)_{L^2(G)},
\end{equation}
and set $\CB:=\CB_\Omega$.
The form $\CB_{G}$ is obviously continuous and the continuity constant is independent of $\omega$ if we use the norm $\|\cdot\|_{\curl, \omega}$.

We now look for $\Vu\in \VH_0(\curl, \Omega)$ such that
\begin{equation}
\label{eq:problem}
\CB(\Vu, \Vpsi)=(\Vf, \Vpsi)_{L^2(\Omega)} \quad\text{for all } \Vpsi\in \VH_0(\curl, \Omega).
\end{equation}
We implicitly assume that the above problem is a multiscale problem, i.e.\ the coefficients $\mu$ and $\varepsilon$ are rapidly varying on a very fine sale.
Fredholm theory guarantees the existence of a unique solution $\Vu$ to \eqref{eq:problem} provided that $\omega$ is not an eigenvalue of curl-curl-operator, which we will assume from now on.
This in particular implies that there is $\gamma(\omega)>0$ such that $\CB$ is inf-sup stable with constant $\gamma(\omega)$, i.e.\
\begin{equation}
\label{eq:infsupanalytic}
\inf_{\Vv\in \VH_0(\curl)\setminus\{0\}}\sup_{\Vpsi\in\VH_0(\curl)\setminus\{0\}}\frac{|\CB(\Vv, \Vpsi)|}{\|\Vv\|_{\curl, \omega}\|\Vpsi\|_{\curl, \omega}}\geq\gamma(\omega).
\end{equation}

\subsection{Mesh and interpolation operator}
\label{subsec:mesh}
Let $\CT_H$ be a regular partition of $\Omega$ into tetrahedra, such that $\cup\CT_H=\overline{\Omega}$ and any two distinct $T, T'\in \CT_H$ are either disjoint or share a common vertex, edge or face.
We assume the partition $\CT_H$ to be shape-regular and quasi-uniform.
The global mesh size is defined as $H:=\max\{ \diam(T)|T\in \CT_{H}\}$.
$\CT_H$ is a coarse mesh in the sense that it does not resolve the fine-scale oscillations of the parameters.

Given any subdomain $G\subset \overline{\Omega}$ define the patches via
\[\UN^1(G):=\UN(G):=\Int(\cup\{T\in \CT_{H}|T\cap\overline{G}\neq \emptyset\})\qquad \text{and}\qquad\UN^m(G):=\UN(\UN^{m-1}(G)).\]
The shape regularity implies that there is a uniform bound $C_{\mathrm{ol}, m}$ on the number of elements in the $m$-th order patch
\[\max_{T\in \CT_{H}}\operatorname{card}\{K\in \CT_{H}|K\subset\overline{\UN^m(T)}\}\leq C_{\mathrm{ol}, m}\]
and the quasi-uniformity implies that $C_{\mathrm{ol}, m}$ depends polynomially on  $m$.
We abbreviate $C_{\mathrm{ol}}:=C_{\mathrm{ol}, 1}$.
We denote the lowest order N{\'e}d{\'e}lec finite element, cf.\ \cite[Section 5.5]{Monk}, by
\[
  \mathring{\CN}(\CT_H):=\{\Vv\in \VH_0(\curl)|\forall T\in \CT_H: \Vv|_T(\Vx)=\Va_T\times\Vx+\Vb_T \text{ with }\Va_T, \Vb_T\in\cz^3\}.
\]

We require an $\VH(\curl)$-stable interpolation operator (with some additional properties) for the numerical homogenization. 
The only suitable candidate is the Falk-Winter interpolation operator, see \cite{FW14}.
Some important properties are summarized below, see \cite{GHV17} for details and proofs.

\begin{proposition}\label{p:proj-pi-H-E}
There exists a projection $\pi_H^E:\VH_0(\curl)\to \mathring{\CN}(\CT_H)$ with the following local stability properties:
For all $\Vv\in \VH_0(\curl)$ and all $T\in \CT_H$ it holds that
\begin{align}
\label{eq:stabilityL2}
\|\pi_H^E(\Vv)\|_{L^2(T)}&\lesssim \bigl(\|\Vv\|_{L^2(\UN(T))}+H\|\curl\Vv\|_{L^2(\UN(T))}\bigr),\\*
\label{eq:stabilitycurl}
\|\curl\pi_H^E(\Vv)\|_{L^2(T)}&\lesssim \|\curl\Vv\|_{L^2(\UN(T))}.
\end{align}
Moreover, for any $\Vv\in \VH_0(\curl, \Omega)$, there are $\Vz\in \VH^1_0(\Omega)$ and $\theta\in H^1_0(\Omega)$ such that
$\Vv-\pi_H^E(\Vv)=\Vz+\nabla \theta$
with the local bounds for every $T\in \CT_H$
\begin{equation}
\label{eq:regulardecomp}
\begin{split}
H^{-1}\|\Vz\|_{L^2(T)}+\|\nabla \Vz\|_{L^2(T)}&\lesssim\|\curl\Vv\|_{L^2(\UN^3(T))},\\
H^{-1}\|\theta\|_{L^2(T)}+\|\nabla \theta\|_{L^2(T)}&\lesssim\bigl(\|\Vv\|_{L^2(\UN^3(T))}+H\|\curl\Vv\|_{L^2(\UN^3(T))}\bigr),
\end{split}
\end{equation}
where $\nabla \Vz$ stands for the Jacobi matrix of $\Vz$.  
\end{proposition}

The stability estimates in particular imply that $\pi_H^E$ is stable with respect to the $\|\cdot\|_{\curl, \omega}$-norm if the condition  $\omega H\lesssim 1$ is fulfilled:
\[\|\pi_H^E\Vv\|_{\curl, \omega}\lesssim \|\Vv\|_{\curl, \omega}\qquad\qquad \text{if }\quad\omega H\lesssim 1.\]

\section{Ideal numerical homogenization}
\label{sec:LODideal}
In this section we introduce an ideal numerical homogenization scheme which approximates the exact solution in $\VH_0(\curl)$ by a coarse part (which itself is a good approximation in $H^{-1}(\Omega)$) and a corrector contribution.
The idea is based on the direct sum splitting $\VH_0(\curl)=\mathring{\CN}(\CT_H)\oplus\VW$ with $\VW:=\ker (\pi_H^E)$ the kernel of the Falk-Winther interpolation operator introduced in the previous section.
The regular decomposition estimates \eqref{eq:regulardecomp} directly imply for any $\Vw\in\VW$
\begin{equation}
\label{eq:dualnormW}
\|\Vw\|_{\VH(\Div)^\prime}\lesssim H\|\Vw\|_{\VH(\curl)}.
\end{equation}

From now on, we assume the resolution condition
\begin{equation}
\label{eq:resolcond}
\omega H\lesssim 1,
\end{equation}
which means that a few degrees of freedom per wavelength are required. 
Under this resolution condition, $\CB$ is stable on $\VW$, as details the next lemma.

\begin{lemma}[Properties of $\VW$]
\label{lem:wellposedW}
Let $\Vw\in \VW$ be decomposed as $\Vw=\Vz+\nabla \theta$ and \eqref{eq:resolcond} be satisfied. Then
\begin{itemize}
\item we have a ($\omega$-independent) norm equivalence between $\|\cdot\|_{\curl, \omega}$ and $\||\Vw|\|^2:=\|\curl \Vz\|^2+\omega^2\|\nabla \theta\|^2$
\item there is $\alpha>0$ independent of $\omega$ such that
\[\inf_{\Vw\in\VW\setminus\{0\}}\sup_{\Vphi\in\VW\setminus\{0\}}\frac{|\CB(\Vw, \Vphi)|}{\|\Vw\|_{\curl, \omega}\|\Vphi\|_{\curl, \omega}}\geq \alpha.\]
\end{itemize}
\end{lemma}
\begin{proof}
For the norm equivalence we obtain using \eqref{eq:regulardecomp} and $\curl\Vw=\curl\Vz$
\begin{align*}
\||\Vw|\|^2=\|\curl\Vz\|^2\!+\!\omega^2\|\nabla\theta\|^2\lesssim \|\curl\Vw\|^2+\omega^2\|\Vw\|^2+\omega^2H^2\|\curl\Vw\|^2\lesssim\|\Vw\|_{\curl, \omega}^2,\\
\|\Vw\|_{\curl, \omega}^2\leq \|\Vz\|_{\curl, \omega}^2+\|\nabla \theta\|^2_{\curl, \omega}\lesssim \|\curl\Vz\|^2\!+\!\omega^2H^2\|\curl\Vz\|^2+\omega^2\|\nabla\theta\|^2\lesssim \||\Vw|\|^2.
\end{align*}
For the inf-sup-constant, set $F(\Vw)=\Vz-\nabla\theta$. 
Observe that $\curl\pi_H^E\Vz=\curl\pi_H^E\Vw=0$ because of the commuting property of $\pi_H^E$. 
Then
\begin{align*}
\Re\{\CB(\Vw, (\id-\pi_H^E)F(\Vw))\}&\gtrsim \|\curl\Vz\|^2+\omega^2\|\nabla\theta\|^2-\omega^2\|\Vz\|^2-2\omega^2|(\varepsilon\Vz, \nabla\theta)|\\
&\quad-2\omega^2|(\varepsilon\Vz, \pi_H^E\Vz)|-2\omega^2|(\varepsilon\nabla \theta, \pi_H^E\Vz)|,
\end{align*}
where we used $\pi_H^E\nabla\theta=-\pi_H^E\Vz$ because of $\pi_H^E\Vw=0$.
Applying Young's inequality, the stability of $\pi_H^E$ \eqref{eq:stabilityL2}--\eqref{eq:stabilitycurl}, estimate \eqref{eq:regulardecomp} and using the resolution condition \eqref{eq:resolcond}, we arrive at
\[\Re\{\CB(\Vw, (\id-\pi_H^E)F(\Vw))\}\gtrsim \|\curl\Vz\|^2+\omega^2\|\nabla\theta\|^2\gtrsim \|\Vw\|_{\curl, \omega}^2\]
because of the norm equivalence.
The estimate $\|(\id-\pi_H^E)F(\Vw)\|_{\curl, \omega}\lesssim \|\Vw\|_{\curl, \omega}$ finally gives the claim.
\end{proof}

In contrast to coercive problems, unique solvability is not guaranteed when $\CB$ is restricted to subspaces.
Therefore, the inf-sup-stability of $\CB$ on $\VW$ is the crucial ingredient to introduce a well-defined Corrector Green's Operator.

\begin{definition}
For $\VF\in \VH_0(\curl)^\prime$, we define the Corrector Green's Operator
\begin{equation}
\label{eq:correctorgreen}
\CG:\VH_0(\curl)^\prime\to\VW \qquad \quad \text{by}\qquad\quad \CB(\CG(\VF), \Vw)=\VF(\Vw)\quad \text{ for all }\Vw\in\VW.
\end{equation}
\end{definition}
\indent
Let $\CL:\VH_0(\curl)\to\VH_0(\curl)^\prime$ denote the differential operator associated with $\CB$ and set $\CK:=-\CG\circ\CL$.
Inspired by the procedure in \cite{GHV17}, an ideal numerical homogenization scheme consists in solving the variational problem over the ``multiscale'' space $(\id+\CK)\mathring{\CN}(\CT_H)$.
The well-posedness of this scheme is proved in the next lemma.

\begin{lemma}
Under the resolution condition, we have with $\gamma(\omega)$ from \eqref{eq:infsupanalytic} that
\begin{equation}
\label{eq:infsupideal}
\inf_{\Vv_H\in \mathring{\CN}(\CT_H)\setminus\{0\}}\sup_{\Vpsi_H\in\mathring{\CN}(\CT_H)\setminus\{0\}}\frac{|\CB((\id+\CK)\Vv_H, (\id+\CK)\Vpsi_H)|}{\|\Vv_H\|_{\curl, \omega}\|\Vpsi_H\|_{\curl, \omega}}\gtrsim\gamma(\omega).
\end{equation}
\end{lemma}
\begin{proof}
Fix $\Vv_H\in\mathring{\CN}(\CT_H)$. From \eqref{eq:infsupanalytic}, there exists $\Vpsi\in\VH_0(\curl)$ with $\|\Vpsi\|_{\curl, \omega}=1$ such that
\[|\CB((\id+\CK)\Vv_H, \Vpsi)|\geq \gamma(\omega) \|(\id+\CK)\Vv_H\|_{\curl, \omega}.\]
By the definition of $\CK$, it holds that $(\id+\CK)\pi_H^E\Vpsi=(\id+\CK)\Vpsi$ and $\CB((\id+\CK)\Vv_H, \Vw)=0$ for all $\Vw\in\VW$.
Thus, we obtain
\begin{align*}
|\CB((\id+\CK)\Vv_H, (\id+\CK)\pi_H^E\Vpsi)|&=|\CB((\id+\CK)\Vv_H, (\id+\CK)\Vpsi)|=|\CB((\id+\CK)\Vv_H, \Vpsi)|\\
&\geq \gamma(\omega) \|(\id+\CK)\Vv_H\|_{\curl, \omega}.
\end{align*}
The claim follows now by the norm equivalence
\[\|\Vv_H\|_{\curl, \omega}=\|\pi_H^E(\id+\CK)\Vv_H\|_{\curl, \omega}\lesssim \|(\id+\CK)\Vv_H\|_{\curl, \omega},\]
which is a result of the stability of $\pi_H^E$.
\end{proof}

Before we introduce the ideal numerical homogenization scheme, we summarize the approximation and stability properties of the Corrector Green's Operator, cf.\ \cite{GHV17}.

\begin{lemma}[Ideal corrector estimates]
\label{lem:correcestimatesideal}
Any $\VF\in \VH_0(\curl)^\prime$ and any $\Vf\in \VH(\Div)$ satisfy
\begin{align}
\label{eq:correcdual}
H\|\CG(\VF)\|_{\curl, \omega}+\|\CG(\VF)\|_{\VH(\Div)^\prime}&\lesssim H\alpha^{-1}\|\VF\|_{\VH_0(\curl)^\prime}\\
\label{eq:correcdiv}
H\|\CG(\Vf)\|_{\curl, \omega}+\|\CG(\Vf)\|_{\VH(\Div)^\prime}&\lesssim H^2\alpha^{-1}\|\Vf\|_{\VH(\Div)}.
\end{align}
\end{lemma}
\indent
Collecting the results of the previous lemmas, we have the following result on our ideal numerical homogenization scheme.

\begin{theorem}
\label{thm:idealLOD}
Let $\Vu$ denote the exact solution to \eqref{eq:problem} and $\Vu_H=\pi_H^E\Vu$. 
Then
\begin{itemize}
\item it holds that $\Vu=\Vu_H+\CK(\Vu_H)+\CG(\Vf)$
\item assuming \eqref{eq:resolcond}, $\Vu_H$ is characterized as the unique solution to
\begin{equation}
\label{eq:idealLOD}
\CB((\id+\CK)\Vu_H, (\id+\CK)\Vpsi_H)=(\Vf, (\id+\CK)\Vpsi_H)\qquad \text{for all }\Vpsi_H\in \mathring{\CN}(\CT_H)
\end{equation}
\item assuming \eqref{eq:resolcond}, it holds that
\begin{equation}
\label{eq:idealapriori}
\|\Vu-(\id+\CK)\Vu_H\|_{\curl, \omega}+\|\Vu-\Vu_H\|_{\VH(\Div)^\prime}\lesssim H\|\Vf\|_{\VH(\Div)}.
\end{equation}
\end{itemize}
\end{theorem}
\begin{proof}
The proof of the first two items carries over verbatim from the elliptic case \cite{GHV17}. 
The a priori error estimate \eqref{eq:idealapriori} follows from the first item and Lemma \ref{lem:correcestimatesideal}.
\end{proof}

The theorem shows that $(\id+\CK)\Vu_H$ approximates the analytical solution with linear rate without assumptions on the regularity of the problem.
What is more, only the reasonable resolution condition $\omega H\lesssim 1$ is required, overcoming the pollution effect.
However, the determination of $\CK$ requires the solution of global problems, which limits the practical usability of the scheme.

\section{Quasi-local numerical homogenization}
\label{sec:LOD}
\subsection{Exponential decay and localized corrector}
The property that $\CK$ can be approximated by local correctors is directly linked to the decay properties of $\CG$ defined in \eqref{eq:correctorgreen}.
The following result states -- loosely speaking -- in which distance (measured in unit of the coarse mesh size $H$) from the support of the source term $\VF$ the weighted $\VH(\curl)$-norm of $\CG(\VF)$ becomes negligibly small.
For that, recall the definition of element patches $\UN^m(T)$ from Section \ref{subsec:mesh}.

\begin{proposition}
\label{prop:decay}
Let $T\in \CT_H$, $m\in\mathbb{N}$ and $\VF_T\in\VH_0(\curl)^\prime$ be a local source functional, i.e.\ $\VF_T(\Vv)=0$ for all $\Vv\in\VH_0(\curl)$ with $\supp(\Vv)\subset\Omega\setminus T$.
If \eqref{eq:resolcond} holds, there exists $0<\tilde{\beta}<1$ such that
\begin{equation}
\label{eq:expdecay}
\|\CG(\VF_T)\|_{\curl, \omega, \Omega\setminus \UN^m(T)}\lesssim \beta^m\|\VF_T\|_{\VH_0(\curl)^\prime}.
\end{equation}
\end{proposition}
\begin{proof}
The proof can be easily adapted from the elliptic case in \cite{GHV17} using the inf-sup-stability of $\CB$ over $\VW$ from Lemma \ref{lem:wellposedW}.
\end{proof}

The result can be used to approximate $\CK$, which has a non-local argument, via
\[\CK(\Vv_H)=-\sum_{T\in\CT_H}\CG(\CL_T(\Vv_H)),\]
where the localized differential operator $\CL_T:\VH(\curl, T)\to\VH(\curl, \Omega)^\prime$ is associated with $\CB_T$, the restriction of $\CB$ to the element $T$.
Proposition \ref{prop:decay} now suggests to truncate the computation of $\CG(\VF_T)$ to the patches $\UN^m(T)$ and then collect the results from all elements $T$.
Typically, $m$ is referred to as \emph{oversampling parameter}.

\begin{definition}[Localized Corrector Approximation]
\label{def:localcorrec}
For any element $T\in\CT_H$ we define its patch $\Omega_T:=\UN^m(T)$.
Let $\VF\in\VH_0(\curl)^\prime$ be the sum of local functionals, i.e.\ $\VF=\sum_{T\in\CT_H}\VF_T$ with $\VF_T$ as in Proposition \ref{prop:decay}.
Denote by $\pi_{H, \Omega_T}^E: \VH_0(\curl, \Omega)\to\mathring{\CN}(\CT_H(\Omega_T))$ the Falk-Winther interpolation operator which enforces essential boundary conditions (i.e.\ zero tangential traces) on $\partial\Omega_T$.
We then define
\begin{equation}
\label{eq:Wlocalnonconf}
\VW(\Omega_T):=\{\Vw\in\VH_0(\curl)| \Vw=0 \text{ outside }\Omega_T, \pi_{H, \Omega_T}^E\Vw=0\}\nsubseteq \VW.
\end{equation}
We call $\CG_{T,m}(\VF_T)\in\VW(\Omega_T)$ the \emph{localized corrector} if it solves
\begin{equation}
\label{eq:localcorrec}
\CB(\CG_{T,m}(\VF_T), \Vw)=\VF_T(\Vw)\qquad \text{for all }\Vw\in\VW(\Omega_T).
\end{equation}
The global corrector approximation is then given by
\[\CG_m(\VF)=\sum_{T\in\CT_H}\CG_{T,m}(\VF_T).\]
\end{definition}
\indent
Observe that problem \eqref{eq:localcorrec} is only formulated on the patch $\Omega_T$.
Its well-posedness can be proved as in Lemma \ref{lem:wellposedW}: For $\Vw\in\VW(\Omega_T)$, use $(\id-\pi_{H,\Omega_T}^E) F(\Vw)\in\VW(\Omega_T)$ as test function.
We emphasize that the definition of $\VW(\Omega_T)$ via $\pi_{H, \Omega_T}^E$ is needed to make this test function a member of $\VW(\Omega_T)$, otherwise the support would be enlarged.
This is a non-conforming definition of the localized corrector (i.e.\ $\pi_H^E\CG_m(\cdot)\neq 0$), so that additional terms appear in the error analysis.
However, the non-conformity error only plays a role near the boundary of $\partial\Omega_T$ and can therefore be controlled very well.

\begin{theorem}
\label{thm:correctorerror}
Let $\CG(\VF)$ be the ideal Green's corrector and $\CG_m(\VF)$ the localized  corrector from Definition \ref{def:localcorrec}.
 Under \eqref{eq:resolcond}, there exists $0<\beta<1$ such that
\begin{align}
\label{eq:truncationerror}
\|\CG(\VF)-\CG_m(\VF)\|_{\curl, \omega}&\lesssim \sqrt{C_{\mathrm{ol}, m}}\,\beta^m\Bigl(\sum_{T\in \CT_H}\|\VF_T\|^2_{\VH_0(\curl)^\prime}\Bigr)^{1/2},\\
\label{eq:nonconferror}
\|\pi_H^E\CG_m(\VF)\|_{\curl, \omega}&\lesssim \sqrt{C_{\mathrm{ol}, m}\,}\beta^m\Bigl(\sum_{T\in \CT_H}\|\VF_T\|^2_{\VH_0(\curl)^\prime}\Bigr)^{1/2}.
\end{align}
\end{theorem}
The proof is postponed to Subsection \ref{subsec:proof}.

\subsection{The quasi-local numerical homogenization scheme}
Following the above motivation, we define a quasi-local numerical homogenization scheme by replacing $\CK$ in the ideal scheme \eqref{eq:idealLOD} with $\CK_m$.

\begin{definition}
Let $\CK_m$ be defined as described in the previous subsection.
The \emph{quasi-local numerical homogenization scheme} seeks $\Vu_{H,m}\in \mathring{\CN}(\CT_H)$ such that
\begin{equation}
\label{eq:LOD}
\CB((\id+\CK_m)\Vu_{H,m}, (\id+\CK_m)\Vv_H)=(\Vf, (\id+\CK_m)\Vv_H)\qquad \text{ for all }\Vv_H\in\mathring{\CN}(\CT_H).
\end{equation}
\end{definition}
\indent
We observe that $\CK_m$ can be computed by solving local decoupled problems, see \cite{GHV17} for details.
Note that the spaces $\VW(\Omega_T)$ are still infinite dimensional so that in practice, we require an additional fine-scale discretization of the corrector problems. 
We omit this step here and refer the reader to \cite{GHV17} for the elliptic case and \cite{Pet17} for the Helmholtz equation.

We now prove the well-posedness and the a priori error estimate for the quasi-local numerical homogenization scheme.

\begin{theorem}[Well-posedness of \eqref{eq:LOD}]
\label{thm:LODwellposed}
If the resolution condition \eqref{eq:resolcond} and the oversampling condition
\begin{equation}
\label{eq:oversamplcond}
m\gtrsim|\log\bigl(\gamma(\omega)/\sqrt{C_{\mathrm{ol}, m}}\bigr)|/|\log(\beta)|
\end{equation}
are fulfilled, $\CB$ is inf-sup-stable over $(\id+\CK_m)\mathring{\CN}(\CT_H)$, i.e.\
\[\inf_{\Vv_H\in\mathring{\CN}(\CT_H)\setminus\{0\}}\sup_{\Vpsi_H\in\mathring{\CN}(\CT_H)\setminus\{0\}}\frac{|\CB((\id+\CK_m)\Vv_H, (\id+\CK_m)\Vpsi_H)|}{\|\Vv_H\|_{\curl, \omega}\|\Vpsi\|_{\curl, \omega}}\geq \gamma_{\mathrm{LOD}}(\omega) \approx\gamma(\omega).\]
\end{theorem}
\begin{theorem}[A priori estimate]
\label{thm:apriori}
Let $\Vu$ denote the analytical solution to \eqref{eq:problem} and $\Vu_{H,m}$ the solution to \eqref{eq:LOD}.
If the resolution condition \eqref{eq:resolcond} and the oversampling condition 
\begin{equation}
\label{eq:oversamplcond1}
m\gtrsim|\log\bigl(\gamma_{\mathrm{LOD}}(\omega)/\sqrt{C_{\mathrm{ol}, m}}\bigr)|/|\log(\beta)|
\end{equation} 
are fulfilled, then 
\begin{equation}
\label{eq:apriori}
\|\Vu-(\id+\CK_m)\Vu_{H,m}\|_{\curl, \omega}\lesssim (H+\beta^m \gamma^{-1}(\omega))\|\Vf\|_{\VH(\Div)}.
\end{equation}
\end{theorem}
\indent
Note that the oversampling condition \eqref{eq:oversamplcond1} is -- up to constants independent of $H$ and $\omega$ -- the same as condition \eqref{eq:oversamplcond}.
Since $C_{\mathrm{ol}, m}$ grows polynomially in $m$ for quasi-uniform meshes, it is satisfiable and depends on the behavior of $\gamma(\omega)$.
If $\gamma(\omega)\lesssim\omega^q$, we derive $m\approx \log(\omega)$, which is a better resolution condition than for a standard discretization.
Note that in \eqref{eq:apriori}, we can replace $\gamma^{-1}(\omega)$ with $C_{\mathrm{stab}}(\omega)$, the stability constant of the original problem \eqref{eq:problem}.
This is exactly the same a priori estimate as for Helmholtz problems in \cite{Pet17}.
To sum up, an oversampling parameter $m\approx |\log(\omega)|$ is sufficient for the stability of the LOD.
Requiring additionally $m\approx|\log(H)|$, we obtain a linear convergence rate for the error.

\subsection{Main proofs}
\label{subsec:proof}
Theorem \ref{thm:correctorerror} results from the exponential decay of $\CG$ in Proposition \ref{prop:decay}.

\begin{proof}[Proof of Theorem \ref{thm:correctorerror}]
We start by proving the following local estimate
\begin{equation}
\label{eq:truncerrorlocal}
\|\CG(\VF_T)-\CG_{T,m}(\VF_T)\|_{\curl, \omega}\lesssim \tilde{\beta}^m\|\VF_T\|_{\VH_0(\curl)^\prime}.
\end{equation}
By Strang's second Lemma we obtain
\begin{align*}
\|\CG(\VF_T)-\CG_{T,m}(\VF_T)\|_{\curl, \omega}&\lesssim \inf_{\Vw_{T,m}\in\VW(\Omega_T)}\|\CG(\VF_T)-\Vw_{T,m}\|_{\curl, \omega}\\
&\quad+\underset{\|\Vphi_{T,m}\|_{\curl, \omega}=1}{\sup_{\Vphi_{T,m}\in\VW(\Omega_T)}}|\CB(\CG(\VF_T), \Vphi_{T,m})-\VF_T(\Vphi_{T,m})|.
\end{align*}
The first term can be estimated as in \cite{GHV17}.
For the second term, we have due to \eqref{eq:correctorgreen} that it is equal to
\[\sup_{\Vphi_{T,m}\in\VW(\Omega_T), \|\Vphi_{T,m}\|_{\curl, \omega}=1}|\CB(\CG(\VF_T), \Vphi_{T,m}-\Vphi)-\VF_T(\Vphi_{T,m}-\Vphi)|\]
for any $\Vphi\in \VW$. 
Fixing $\Vphi_{T,m}=\Vz_{T, m}+\nabla \theta_{T,m}$, we choose $\Vphi=(\id-\pi_H^E)(\eta\Vz_{T,m}+\nabla(\eta\theta_{T,m}))$ with a cut-off function such that $\Vphi_{T,m}-\Vphi=0$ in $\UN^{m-2}(T)$.
Then $\VF_T(\Vphi_{T,m}-\Vphi)=0$ and we get with the stability of $\pi_H^E$ and \eqref{eq:regulardecomp}
\begin{align*}
|\CB(\CG(\VF_T), \Vphi_{T,m}-\Vphi)|&\lesssim \|\CG(\VF_T)\|_{\curl, \omega, \Omega\setminus \UN^{m-2}(T)}\|\Vphi_{T,m}-\Vphi\|_{\curl, \omega}\\
&\lesssim \|\CG(\VF_T)\|_{\curl, \omega, \Omega\setminus \UN^{m-2}(T)}.
\end{align*}
Combination with Proposition \ref{prop:decay} gives \eqref{eq:truncerrorlocal}.

For \eqref{eq:truncationerror}, we split the error as
\[\|\CG(\VF)-\CG_m(\VF)\|_{\curl, \omega}\leq \|(\id-\pi_H^E)(\CG(\VF)-\CG_m(\VF))\|_{\curl, \omega}+\|\pi_H^E\CG_m(\VF)\|_{\curl, \omega}.\]
The first term can be estimated with the procedure from \cite{GHV17}.
The second term is the left-hand side of \eqref{eq:nonconferror} and thus, it suffices to prove \eqref{eq:nonconferror}.
We observe that $\pi_H^E\CG_{T,m}(\VF_T)\neq 0$ only on a small ring $R\subset \UN^{m+1}(T)$ because $\pi_H^E$ and $\pi_{H, \Omega_T}^E$ only differ near the boundary of $\Omega_T$.
Hence, we get
\begin{align*}
\|\pi_H^E\CG_m(\VF)\|^2_{\curl, \omega}&\leq\sum_{T\in\CT_H}|(\pi_H^E\CG_m(\VF), \pi_H^E\CG_{T,m}(\VF_T))_{\curl, \omega}|\\
&\lesssim \sum_T\|\pi_H^E\CG_m(\VF)\|_{\curl, \omega, \UN^{m+1}(T)}\|\pi_H^E(\CG(\VF_T)-\CG_{T, m}(\VF_T))\|_{\curl, \omega}\\
&\lesssim \sqrt{C_{\mathrm{ol}, m}}\|\pi_H^E\CG_m(\VF)\|_{\curl, \omega}\Bigl(\sum_T\|\CG(\VF_T)-\CG_{T,m}(\VF_T)\|_{\curl, \omega}\Bigr)^{1/2}.
\end{align*}
Application of \eqref{eq:truncerrorlocal} gives the claim.
\end{proof}

The well-posedness of the quasi-local numerical scheme comes from the well-posedness of the ideal scheme (Theorem \ref{thm:idealLOD}) and the fact that the localized corrector is exponentially close to the ideal corrector.

\begin{proof}[Proof of Theorem \ref{thm:LODwellposed}]
Fix $\Vv_H\in\mathring{\CN}(\CT_H)$ and set $\tilde{\Vv}_H=\pi_H^E(\id+\CK_m)(\Vv_H)$.
According to Theorem \ref{thm:idealLOD}, there exists $\Vpsi_H\in\mathring{\CN}(\CT_H)$ with $\|\Vpsi_H\|_{\curl, \omega}=1$ such that
\[|\CB((\id+\CK)\tilde{\Vv}_H, (\id+\CK)\Vpsi_H)|\geq \gamma(\omega)\|\tilde{\Vv}_H\|_{\curl,\omega}.\]
As $\CB(\Vw, (\id+\CK)\Vpsi_H)=0$ for all $\Vw\in\VW$, we derive
\begin{align*}
\CB((\id+\CK_m)\Vv_H, (\id+\CK)\Vpsi_H)&=\CB((\id\!+\CK_m)\Vv_H\!-\!(\id-\pi_H^E)((\id+\CK_m)\Vv_H), (\id+\CK)\Vpsi_H)\\
&=\CB(\tilde{\Vv}_H, (\id+\CK)\Vpsi_H)=\CB((\id+\CK)\tilde{\Vv}_H, (\id+\CK)\Vpsi_H).
\end{align*}
This yields together with Theorem \ref{thm:correctorerror}
\begin{align*}
&\!\!\!\!|\CB((\id+\CK_m)\Vv_H, (\id+\CK_m)\Vpsi_H)|\\
&=|\CB((\id+\CK_m)\Vv_H, (\CK_m-\CK)\Vpsi_H)+\CB((\id+\CK)\Vv_H, (\id+\CK)\Vpsi_H)|\\
&=|\CB((\id+\CK_m)\Vv_H, (\CK_m-\CK)\Vpsi_H)+\CB((\id+\CK)\tilde{\Vv}_H, (\id+\CK)\Vpsi_H)|\\
&\geq \gamma(\omega)\|\tilde{\Vv}_H\|_{\curl, \omega}-C\sqrt{C_{\mathrm{ol}, m}}\,\beta^m\|(\id+\CK_m)\Vv_H\|_{\curl, \omega}.
\end{align*}
Moreover, we have 
\[\|(\id+\CK_m)\Vv_H\|_{\curl, \omega}\lesssim (1+\beta^m)\|\Vv_H\|_{\curl, \omega}\lesssim \|\Vv_H\|_{\curl, \omega},\]
since $\beta<1$, and
\begin{align*}
\|\Vv_H\|_{\curl, \omega}&=\|\pi_H^E(\id+\CK)\Vv_H\|_{\curl, \omega}=\|\pi_H^E(\id+\CK_m)\Vv_H+\pi_H^E(\CK-\CK_m)\Vv_H\|_{\curl, \omega}\\
&\lesssim \|\tilde{\Vv}_H\|_{\curl, \omega}+C\sqrt{C_{\mathrm{ol}, m}}\,\beta^m\|\Vv_H\|_{\curl, \omega}.
\end{align*}
If $m$ is large enough (indirectly implied by the oversampling condition), the second term can be hidden on the left-hand side.
Thus, we finally obtain
\[|\CB((\id+\CK_m)\Vv_H, (\id+\CK_m)\Vpsi_H)|\gtrsim (\gamma(\omega)-C\sqrt{C_{\mathrm{ol}, m}}\,\beta^m)\|\Vv_H\|_{\curl, \omega}.\]
Application of the oversampling condition \eqref{eq:oversamplcond} gives the assertion.
\end{proof}

The proof of the a priori error estimate is inspired by the procedure for the Helmholtz equation \cite{Pet17} and uses duality arguments.

\begin{proof}[Proof of Theorem \ref{thm:apriori}]
Denote by $\Ve$ the error $\Vu-(\id+\CK_m)\Vu_{H,m}$ and set $\Ve_{H, m}:=(\id+\CK_m)\pi_H^E(\Ve)$.
Let $\Vz_H\in\mathring{\CN}(\CT_H)$ be the solution to the dual problem
\begin{equation*}
\CB((\id+\CK_m)\Vv_H, (\id+\CK_m)\Vz_H)=(\Ve_{H,m}, (\id+\CK_m)\Vv_H)_{\curl, \omega}\qquad \text{for all }\Vv_H\in\mathring{\CN}(\CT_H).
\end{equation*}
Using the fact that $\CB(\Vw, (\id+\CK)\Vz_H)=0$ for all $\Vw\in\VW$ and employing the Galerkin orthogonality $\CB(\Ve, (\id+\CK_m)\Vz_H)=0$, we obtain that
\begin{align*}
\|\Ve_{H,m}\|^2_{\curl, \omega}&=\CB(\Ve_{H,m}, (\id+\CK_m)\Vz_H)\\
&=\CB(\Ve_{H,m}, (\CK_m-\CK)\Vz_H)+\CB(\Ve_{H,m}, (\id+\CK)\Vz_H)\\
&=\CB(\Ve-\Ve_{H,m}, (\CK-\CK_m)\Vz_H)-\CB(\pi_H^E(\Ve-\Ve_{H,m}), (\id+\CK)\Vz_H).
\end{align*}
Observe that $\pi_H^E(\Ve-\Ve_{H,m})=\pi_H^E\CK_m\pi_H^E(\Ve)$.
Theorem \ref{thm:correctorerror} and \ref{thm:LODwellposed} yield
\begin{align*}
\|\Ve_{H,m}\|_{\curl, \omega}^2&\lesssim \sqrt{C_{\mathrm{ol}, m}}\,\beta^m\|\Ve-\Ve_{H, m}\|_{\curl, \omega}\|\Vz_H\|_{\curl, \omega}\!+\!\sqrt{C_{\mathrm{ol}, m}}\,\beta^m\|\Ve\|_{\curl, \omega}\|\Vz_H\|_{\curl, \omega}\\
&\lesssim \sqrt{C_{\mathrm{ol}, m}}\,\beta^m\, \gamma_{\mathrm{LOD}}^{-1}(\omega)\, (\|\Ve-\Ve_{H, m}\|_{\curl, \omega}+\|\Ve\|_{\curl, \omega})\|\Ve_{H, m}\|_{\curl, \omega}.
\end{align*}
The triangle inequality gives
\[\|\Ve\|_{\curl, \omega}\leq \|(\id-\pi_H^E)(\Ve-\Ve_{H,m})\|_{\curl, \omega}+\|\pi_H^E(\Ve-\Ve_{H,m})\|_{\curl, \omega}+\|\Ve_{H,m}\|_{\curl, \omega}.\]
The above computations and \eqref{eq:nonconferror} imply with the resolution condition \eqref{eq:oversamplcond1}
\[\|\Ve\|_{\curl, \omega}\lesssim \|(\id-\pi_H^E)(\Ve-\Ve_{H,m})\|_{\curl, \omega}.\]

Observe that $\Ve-\Ve_{H,m}=\Vu-(\id+\CK_m)\pi_H^E(\Vu)-(\id+\CK_m)\pi_H^E\CK_m\Vu_{H,m}$.
Since $(\id-\pi_H^E)(\Ve-\Ve_{H,m})\in\VW$, Lemma \ref{lem:wellposedW} gives $\Vw\in\VW$ with $\|\Vw\|_{\curl, \omega}=1$ such that
\begin{align*}
&\!\!\!\!\|(\id-\pi_H^E)(\Ve-\Ve_{H,m})\|_{\curl, \omega}\\
&\lesssim |\CB((\id-\pi_H^E)(\Ve-\Ve_{H, m}), \Vw)|\\
&=|\CB(\Vu, \Vw)\!-\!\CB((\id\!+\CK_m)\pi_H^E\Vu, \Vw)\!-\!\CB((\id\!+\CK_m)\pi_H^E\CK_m\Vu_{H,m}, \Vw)\!-\!\CB(\pi_H^E\CK_m\pi_H^E\Ve, \Vw)|\\
&=|(\Vf, \Vw)\!-\!\CB((\CK_m-\!\CK)\pi_H^E\Vu, \Vw)\!-\!\CB((\CK_m-\!\CK)\pi_H^E\CK_m\Vu_{H,m}, \Vw)\!-\!\CB(\pi_H^E\CK_m\pi_H^E\Ve, \Vw)|.
\end{align*}
Theorems \ref{thm:correctorerror} and \ref{thm:LODwellposed} now give together with the stability of $\pi_H^E$ and \eqref{eq:dualnormW}
\begin{align*}
&\!\!\!\!\|(\id-\pi_H^E)(\Ve-\Ve_{H,m})\|_{\curl, \omega}\\
&\lesssim \bigl(H+\sqrt{C_{\mathrm{ol}, m}}\,\beta^m\gamma^{-1}(\omega)+C_{\mathrm{ol}, m}\,\beta^{2m}\gamma_{\mathrm{LOD}}^{-1}(\omega)\bigr)\|\Vf\|_{\VH(\Div)}+\sqrt{C_{\mathrm{ol}, m}}\,\beta^m\|\Ve\|_{\curl, \omega}.
\end{align*}
The last term can be hidden on the left-hand side and the third term can be absorbed in the second term.
\end{proof}

\section*{Conclusion}
In this paper, we presented and analyzed a numerical homogenization scheme for indefinite $\VH(\curl)$-problems, inspired by \cite{GHV17}.
We showed that the indefinite bilinear form is inf-sup-stable for $\omega H\lesssim 1$ over the kernel of the Falk-Winther interpolation operator, which is crucial for the analysis.
Under this reasonable resolution condition and the additional oversampling condition $m\approx |\log(\gamma(\omega))|$, the numerical homogenization method is stable and yields linear convergence (w.r.t.\ the mesh size) of the error in the $\VH(\curl)$-norm.
These conditions are similar for the Helmholtz equation, suggesting that they are optimal.
Incorporating impedance boundary conditions is subject of future research.

\section*{Acknowledgments}
Financial support by the DFG through project OH 98/6-1 is gratefully acknowledged.
The author would like to thank Dietmar Gallistl (KIT) for fruitful discussion on the subject, in particular for suggesting the non-conforming definition of the localized spaces.
Main ideas of this contribution evolved while the author enjoyed the kind hospitality of the Hausdorff Research Institute for Mathematics (HIM) in Bonn during the trimester program on multiscale problems.

\end{document}